\documentclass[a4paper]{article} 

\newcommand{\f}{\frac}

\newcommand{\del}{\partial}

\newcommand{\im}{\operatorname{Im}}

\newcommand{\R}{\mathbb R}
\newcommand{\C}{\mathbb C}
\newcommand{\N}{\mathbb N}

\newcommand{\re}{\operatorname{Re}}
\newcommand{\eps}{\varepsilon}
\renewcommand{\epsilon}{\varepsilon}

\newcommand{\Om}{\Omega}

\newcommand{\supp}{\operatorname{supp}}
\newcommand{\dist}{\operatorname{dist}}

\newcommand{\dom}{\operatorname{dom}}
\newcommand{\h}{\mathcal{H}}

\newenvironment{enuma}{\begin{enumerate}[(a)]}{\end{enumerate}}
\newenvironment{enumi}{\begin{enumerate}[(i)]}{\end{enumerate}}

\AtBeginDocument{

}

\usepackage{etex}
\usepackage[english]{babel}
\usepackage{tikz}
\usetikzlibrary{arrows,decorations.pathmorphing,backgrounds,positioning,fit,petri,patterns}
\usepackage{pgfplots}
\usepackage[utf8]{inputenc}
\usepackage{graphicx}
\usepackage[colorlinks=true,linkcolor=black,citecolor=black]{hyperref}
\usepackage{amsfonts}
\usepackage{enumerate}
\usepackage{booktabs}
\usepackage{array} 
\usepackage{paralist} 
\usepackage{subfig} 
\usepackage{mathtools}
\usepackage{tabu}
\usepackage{amsthm}
\usepackage{empheq}
\usepackage{amsopn}
\usepackage{dsfont}
\usepackage{wrapfig}
\usepackage[font=small]{caption}
\usepackage{units}
\usepackage[symbol]{footmisc}
\usepackage{todonotes}
\usepackage{amssymb}
\usepackage{pdfpages}
\usepackage{setspace}

\usepackage{kvsetkeys}% provides \comma@parse
\usepackage{etexcmds}% provides \etex@unexpanded

\makeatletter
\newcount\arg@count
\newcommand{\arg@parser}[1]{%
  \advance\arg@count\@ne
  \expandafter\let\csname arg\romannumeral\arg@count\endcsname\comma@entry
}
\newcommand\res[1]{% Default is empty and will be configured later
  \arg@count=\z@
  \comma@parse{ \lambda,A }\arg@parser % Default values
  \arg@count=\z@
  \comma@parse{#1}\arg@parser
  \ifnum\arg@count>2 %
    \@latex@error{Too many arguments}{%
      The macro \string\mycmd\space got \the\arg@count\space
       arguments,\MessageBreak
      but expected are 2 arguments.\MessageBreak
      \@ehd
    }%
  \fi
  \edef\process@me{%
    \noexpand\@res
    {\etex@unexpanded\expandafter{\argi}}%
    {\etex@unexpanded\expandafter{\argii}}%
  }%
  \process@me
}
\newcommand{\@res}[2]{%
  \ensuremath\left( #1 - #2 \right)^{-1}
}
\makeatother

\makeatletter
\newcount\arg@count
\newcommand\p[1]{% Default is empty and will be configured later
  \arg@count=\z@
  \comma@parse{  }\arg@parser % Default values
  \arg@count=\z@
  \comma@parse{#1}\arg@parser
  \ifnum\arg@count=2 %
  \else
    \@latex@error{Wrong number of mandatory arguments}{%
      The macro \string\p\space got \the\numexpr\arg@count-2\relax\space
      mandatory arguments,\MessageBreak
      but expected are 3 mandatory arguments.\MessageBreak
      \@ehd
    }%
  \fi
  \edef\process@me{%
    \noexpand\@p
    {\etex@unexpanded\expandafter{\argi}}%
    {\etex@unexpanded\expandafter{\argii}}%
  }%
  \process@me
}
\newcommand{\@p}[2]{%
  \ensuremath \left\langle #1 , #2 \right\rangle
}
\makeatother

\allowdisplaybreaks
\numberwithin{equation}{section}

\theoremstyle{definition}
\newtheorem{de}{Definition}[section]

\theoremstyle{plain}
\newtheorem{prop}[de]{Proposition}
\newtheorem{lemma}[de]{Lemma}
\newtheorem{theorem}[de]{Theorem}

\newtheorem{corollary}[de]{Corollary}

\theoremstyle{remark}
\newtheorem{remark}[de]{Remark}

\newcommand{\limess}{\sigma_e\big((T_n)_{n\in\N}\big)}
\newcommand{\limpseu}[1]{\Lambda_{e,#1}\big((T_n)_{n\in\N}\big)}
\newcommand{\Tn}{\big((T_n)_{n\in\N}\big)}
\newcommand{\SCI}{\operatorname{SCI}}
\newcommand{\diam}{\operatorname{diam}}

\title{On The Solvability Complexity Index for Unbounded Selfadjoint Operators and Schr\"odinger Operators}
\author{Frank R\"osler\thanks{School of Mathematics\newline 
Cardiff University, \newline
Senghennydd Road, \newline
Cardiff, \newline
CF24 4AG, \newline
United Kingdom\newline
Email: \href{mailto:roslerf@cardiff.ac.uk}{RoslerF@cardiff.ac.uk}}}
\date\today

\begin{document}

\maketitle

\begin{abstract}
	We study the solvability complexity index (SCI) for unbounded selfadjoint operators on separable Hilbert spaces and perturbations thereof. In particular, we show that if the extended essential spectrum of a selfadjoint operator is convex, then the SCI for computing its spectrum is equal to 1. This result is then extended to relatively compact perturbations of such operators and applied to Schr\"odinger operators with compactly supported (complex valued) potentials to obtain SCI=1 in this case, as well.
\end{abstract}

\section{Introduction}

The problem of computing spectra of partial differential operators is fundamental to many problems in physics with real world applications. Perhaps one of the most prominent examples of this is quantum mechanics, where the possible bound state energies of a particle subject to a force described by a potential function $V$ are given by the eigenvalues of the corresponding Schr\"odinger operator $-\Delta+V$. Generically, the spectral problem of such an operator cannot be solved explicitly and one has to resort to numerical methods. By practical constraints, any computer algorithm, which might be used to compute the spectrum, will only be able to handle a finite amount of information about the operator and perform a finite number of arithmetic operations on this information (in practice, this ``finite amount of information'' is usually given by some sort of discretisation of the domain, which approximates the infinite dimensional spectral problem by a finite dimensional one). In other words, any algorithm will always ``ignore'' an infinite amount of information about the operator. One might hope that by increasing the dimension of the approximation (or decreasing the step size of the discretisation), one will eventually obtain a reasonable approximation of the spectrum. Hence, it is a legitimate question to ask:
\begin{center}
	\textit{Given a class of operators $\Om$, does there exist a sequence of algorithms $\Gamma_n$ such that $\Gamma_n(T)\to\sigma(T)$ (in an appropriate sense) for all $T\in\Om$?}
\end{center}
It turns out that the answer to the above question is not always in the affirmative. Indeed, it has been shown in \cite{AHS} that if $\Om=L(\h)$ (the space of bounded operators on a separable Hilbert space $\h$), then for any sequence of algorithms there exists $T\in\Om$ whose spectrum is not approximated by that sequence.
This observation has led to the wider definition of the so-called \emph{Solvability Complexity Index} (SCI), introduced in \cite{Hansen11}, of which we will now give a brief review. 
\begin{de}\label{def:computational_problem}
	A \emph{computational problem} is a quadruple $(\Om,\Lambda,\Xi,\mathcal M)$, where 
	\begin{enumi}
		\item $\Om$ is a set, called the \emph{primary set},
		\item $\Lambda$ is a set of complex valued functions on $\Om$, called the \emph{evaluation set},
		\item $\mathcal M$ is a metric space,
		\item $\Xi:\Om\to M$ is a map, called the \emph{problem function}.
	\end{enumi}
\end{de}
In the above definition, $\Om$ is the set of objects that give rise to the computational problem, $\Lambda$ plays the role of providing the information accessible to the algorithm, and $\Xi:\Om\to \mathcal M$ gives the quantity that one wishes to compute numerically.

An example of a computational problem in the sense of Definition \ref{def:computational_problem} is given by the spectral problem discussed above. Indeed, given a separable Hilbert space $\h$ with orthonormal basis $\{e_i\}$, one can choose $\Om=L(\h)$, $\mathcal M = \{\text{compact subsets of }\C\}$, equipped with the Hausdorff metric, and $\Xi(T)=\sigma(T)$. For the evaluation set one could choose $\Lambda:=\{f_{ij}\,|\,i,j\in\N\}$, where $f_{ij}(T) = \langle Te_i,e_j\rangle$ give the matrix elements of an operator with respect to the basis $\{e_i\}$. 
\begin{de}\label{def:Algorithm}
	Let $(\Om,\Lambda,\Xi,\mathcal M)$ be a computational problem. An \emph{arithmetic algorithm} is a map $\Gamma:\Om\to\mathcal M$ such that for each $T\in\Om$ there exists a finite subset $\Lambda_\Gamma(T)\subset\Lambda$ such that
	\begin{enumi}
		\item the action of $\Gamma$ on $T$ depends only on $\{f(T)\}_{f\in\Lambda_\Gamma(T)}$,
		\item for every $S\in\Om$ with $f(T)=f(S)$ for all $f\in\Lambda_\Gamma(T)$ one has $\Lambda_\Gamma(S)=\Lambda_\Gamma(T)$,
		\item the action of $\Gamma$ on $T$ consists of performing only finitely many arithmetic operations on $\{f(T)\}_{f\in\Lambda_\Gamma(T)}$.
	\end{enumi}
\end{de}
We will refer to any arithmetic algorithm simply as an \emph{algorithm} from now on. For more general concepts the reader may consult \cite{AHS}.

In \cite{AHS} it has been shown that if $\Om$ is the set of compact operators on a separable Hilbert space $\h$, then there exists a sequence of algorithms $\Gamma_n:\Om\to\C$ such that $\Gamma_n(T)\to\sigma(T)$ (in Hausdorff sense) for all $T\in\Om$, while for the set of bounded selfadjoint operators $\Om=\{T\in L(\h)\,|\,T^*=T\}$ this is not possible. 

However, it turns out that there exists a family $\Gamma_{mn}$ of algorithms such that 
$$\lim_{n\to\infty}\lim_{m\to\infty}\Gamma_{mn}(T) = \sigma(T)$$
for all bounded selfadjoint operators. Hence, it \emph{is} possible to compute the spectrum of non-compact operators using algorithms, but the number of limits required may increase (this general phenomenon has first been observed by Doyle and McMullen in the context of finding zeros of polynomials, cf. \cite{DM}). In order to capture this phenomenon, the following definition has been made
\begin{de}[{\cite{AHS}}]\label{def:Tower}
	Let $(\Om,\Lambda,\Xi,\mathcal M)$ be a computational problem. A \emph{tower of algorithms} of height $k$ is a family $\Gamma_{n_1,n_2,\dots,n_k}:\Om\to\mathcal M$ of arithmetic algorithms such that for all $T\in\Om$
	\begin{align*}
		\Xi(T) = \lim_{n_k\to\infty}\cdots\lim_{n_1\to\infty}\Gamma_{n_1,n_2,\dots,n_k}(T).
	\end{align*}
\end{de}
The examples above show that the number of limits required to compute the problem function $\Xi$ is a measure for the numerical complexity of the underlying computational problem. This motivates the
\begin{de}[{\cite{AHS}}]
	A computational problem $(\Om,\Lambda,\Xi,\mathcal M)$ is said to have \emph{Solvability Complexity Index} $k$ if $k$ is the smallest integer for which there exists a tower of algorithms of height $k$ that computes $\Xi$.
	
	If a computational problem has solvability complexity index $k$, we write $\SCI(\Om,\Lambda,\Xi,\mathcal M)=k$. 
\end{de}
\begin{remark}
	In this article we are mainly interested in the spectral problem and will therefore write $\SCI(\Om,\Lambda)$ instead of $\SCI(\Om,\Lambda,\Xi,\mathcal M)$, where it is understood that $\Xi(T)=\sigma(T)$ and $\mathcal M$ is the set of closed subsets of $\C$ equipped with the Attouch-Wets metric $d_{\text{AW}}$ defined as
\begin{align*}
	d_{\text{AW}}(A,B) = \sum_{i=1}^\infty 2^{-i}\min\left\{ 1\,,\,\sup_{|x|<i}\left| \dist(x,A) - \dist(x,B) \right| \right\}.
\end{align*} 
(Note that if $A,B\subset\C$ are bounded, then $d_{\text{AW}}$ coincides with the Hausdorff distance.)
\end{remark}
\noindent In practice it is often important to have explicit estimates on the error $d\big(\Gamma_{n_1,\dots,n_k}(T),\Xi(T)\big)$ for all $T\in\Om$. It is straightforward to show, however, that such an estimate is impossible to obtain as soon as $\SCI(\Om,\Lambda,\Xi,\mathcal M)>1$ (cf. \cite[Thm. 6.1]{AHS}). Indeed, it is easy to see that if for a tower of algorithms $\Gamma_{n_1,\dots,n_k}$ there exist subsequences $n_1(m),\dots,n_k(m)$ such that $\Gamma_{n_1(m),\dots,n_k(m)}(T)<\f1m$ for all $T\in\Om$, then $\tilde\Gamma_m:=\Gamma_{n_1(m),\dots,n_k(m)}$ is in fact a tower of height 1 for $\Om$ and hence $\SCI(\Om)=1$.

For this reason, it is of particular interest to find classes $\Om$ of operators for which $\SCI(\Om,\Lambda,\sigma(\cdot))=1$ (with appropriately chosen $\Lambda$). The present article addresses precisely this question. In fact, we will show that for selfadjoint operators whose \emph{extended essential spectrum} (see \eqref{eq:extended_essential_spectrum}) is convex, we have $\SCI=1$. This is done by explicitly constructing a sequence of arithmetic algorithms which computes the spectrum of any such operator. The result is then extended to certain relatively compact perturbations of such operators. We stress that the new aspect of our work is to consider the \emph{shape of the essential spectrum} as a relevant criterion for reducing the numerical complexity of the spectral problem.
As an application of this approach, we will show that our results apply to non-selfadjoint Schr\"odinger operators with certain well behaved potentials.

The problem of determining the SCI for spectral problems has previously been studied in \cite{Hansen11,AHS} for operator in abstract Hilbert spaces, as well as for partial differential operators. Previous results include
\paragraph{Bounded operators:}
Let $\h,\,\Lambda$ be as in the example above Definition \ref{def:Algorithm}. It was shown in \cite[Th. 3.3, Th. 3.7]{AHS} that then 
\begin{align*}
	\SCI(\Om,\sigma(\cdot)) &= 3\quad\text{ if } \Om = L(\h) \\
	\SCI(\Om,\sigma(\cdot)) &= 2\quad\text{ if } \Om = \{T\in L(\h)\,|\,T\text{ selfadjoint}\} \\
	\SCI(\Om,\sigma(\cdot)) &= 1\quad\text{ if } \Om = \mathcal K(\h),
\end{align*}
where $\mathcal K(\h)$ denotes the set of compact operators. The last bound $\SCI(\mathcal K(\h),\sigma(\cdot)) = 1$ is related to the fact that compact operators can be approximated in operator norm by finite range operators.
\paragraph{Schr\"odinger operators:}
In \cite{AHS}, the SCI for the spectral problem of Schr\"odinger operators with complex valued potentials $V$ has been studied. It has been shown that if 
\begin{equation}\label{eq:AHS_Schroedinger}
 	\Om = \left\{-\Delta+V\,|\,V \text{ is sectorial and }|V(x)|\to\infty\text{ as }|x|\to\infty\right\},
 \end{equation}
 then $\SCI(\Om,\sigma(\cdot))=1$. The proof relies on the fact that operators as in \eqref{eq:AHS_Schroedinger} have compact resolvent. 
 
 In the case of \emph{bounded} potentials, one lacks compact resolvent and the situation is somewhat more difficult. It has been shown in \cite[Th. 4.2]{AHS} that if $\Om$ denotes the set of Schr\"odinegr operators on $\R^d$ with $V$ bounded and of bounded variation, then $\SCI(\Om,\sigma(\cdot))\leq 2$. It has since then been an open problem, whether without any additional information the SCI of this problem is equal to one or two.
 
 The SCI of certain unbounded operators in separable Hilbert spaces, whose matrix representation is banded, has been studied in \cite{Hansen11}.
 
 In this article, we will take a step towards closing this gap. We will prove that if $M>0$ and $\Om$ denotes the set of all Schr\"odinger operators $-\Delta+V$ with $\supp(V)\subset B_M(0)$ and $|\nabla V|\leq M$, then $\SCI(\Om,\sigma(\cdot))=1$ (for the precise statement, see Section \ref{sec:Schroedinger}). This is done by first proving two abstract theorems about the SCI of selfadjoint operators which are of independent interest. The main theorems of this article are Theorems \ref{th:mainth}, \ref{th:perturbation} and \ref{th:Schroedinger}.
 
 The question as to wether the assumption on the support of $V$ is essential for having $\SCI=1$ remains an interesting open problem and will be addressed in future work.

\section{Selfadjoint Operators}\label{sec:selfadjoint}

Let $\h$ be a separable Hilbert space and $\h_n\subset\h$ be a sequence of finite dimensional subspaces such that $\h_n\subset\h_{n+1}$ for all $n\in\N$ and $P_n\xrightarrow{s}I$, where $P_n$ denotes the orthogonal projection onto $\h_n$. Define

\begin{align}\label{eq:Omega}
	\Omega_1:=\Bigg\{T:\dom(T)\to\h \;\Bigg|\; \parbox{4.4cm}{$T \text{ selfadjoint, }\widehat\sigma_{e}(T)\text{ convex}\\[1mm]
	 \text{and }\bigcup_{n\in\N}\h_n\text{ is a core of }T$} \Bigg\},
\end{align}
%\begin{align*}
%	\Omega:=\bigg\{T:\dom(T)\to\h \,\Big|\, T &\text{ selfadjoint, }\dom(T^2)\text{ dense in }\h,\\
%	&\widehat\sigma_e(T)\text{ convex} \text{ and }\bigcup_{n\in\N}\h_n\subset\dom(T) \bigg\},
%\end{align*}
where 
\begin{equation}\label{eq:extended_essential_spectrum}
\widehat\sigma_e(T)=\sigma_{e2}(T)\cup\begin{cases}
	\{+\infty\}, &\text{if }T\text{ unbounded above}\\
	\{-\infty\}, &\text{if }T\text{ unbounded below}
\end{cases}
\end{equation}
and
\begin{align*}
	\sigma_{e2}(T)=\{\lambda\in \C\,|\, \exists (x_k)\subset\dom(T) : \|x_k\|=1\,\forall k,\, x_k\rightharpoonup 0,\, \|(T-\lambda)x_k\|\to 0\}.
\end{align*}
Furthermore, for each $n\in\N$, let $\{e_1^{(n)},\dots,e_{k_n}^{(n)}\}$ be an orthonormal basis of $\h_n$ and define 
\begin{align}\label{eq:def_Lambda_1}
	\Lambda_1 := \big\{ f_{i,j,n} \,|\, 1\leq i,j\leq k_n,\,n\in\N \big\},
\end{align}
where $f_{i,j,n}:T\mapsto \big\langle Te_i^{(n)},e_j^{(n)}\big\rangle$ are the evaluation functions producing the $(i,j)$th matrix elements.
This is the set of information accessible to the algorithm.

\begin{theorem}\label{th:mainth}
	We have $\SCI(\Om_1,\Lambda_1,\sigma(\cdot))=1$.
\end{theorem}
Note that Theorem \ref{th:mainth} in particular applies to bounded selfadjoint operators with convex essential spectrum. In this sense, Theorem \ref{th:mainth} can be viewed as an extension of \cite[Th. 3.7]{AHS}, where it was shown that $\SCI=1$ for the set of all \emph{compact} operators (which naturally satisfy $\sigma_e(T)\subset\{0\}$).

\subsection{Proof of Theorem \ref{th:mainth}}

Define the truncated operator
\begin{equation}\label{eq:Tndef}
	T_n:=P_nT|_{\h_n}.
\end{equation}
This operator can be represented by a finite dimensional (square) matrix with elements $(T_n)_{ij}=\big\langle Te_i^{(n)},e_j^{(n)}\big\rangle$.  Moreover, let $G_n:=\f1n(\mathbb Z+i\mathbb Z)\cap B_n(0)\subset\C.$ 
\begin{lemma}\label{lemma:finite_testing}
	Let $\lambda\in G_n$ and denote by $s(\cdot)$ the smallest singular value of a matrix. Then
	\begin{enumi}
		\item For all $n$ and $\lambda$, we have $s(T_n-\lambda)=\|(T_n-\lambda)^{-1}\|^{-1}_{L(\h_n)}$.
		\item For any $q>0$, testing whether $s(T_n-\lambda)>q$ requires only finitely many arithmetic operations on the matrix elements of $T_n$.
	\end{enumi}
	with the convention that $\|(T_n-\lambda)^{-1}\|^{-1}=0$ for $\lambda\in\sigma(T_n)$.
\end{lemma}
\begin{proof}
	Part (i) was proved in \cite{Hansen11}, while part (ii) follows by noting that $s(T_n-\lambda)>q$ is equivalent to $(T_n-\lambda)^*(T_n-\lambda)-q^2I$ being positive definite; see \cite[Prop. 10.1]{AHS} for a full proof.
\end{proof}
For $n\in\N$ we define a map $\Gamma_n^{(1)}:\Om_1\to\{\text{closed subsets of }\C\}$ by
\begin{align*}
	\Gamma_n^{(1)}(T):=\left\{\lambda\in G_n\,\Big|\,s(T_n-\lambda)\leq\f1n\right\}.
\end{align*}
Then, by the above lemma, each $\Gamma_n^{(1)}$ is an arithmetic tower of height one in the sense of Definition \ref{def:Tower}. Clearly, $\Gamma_n^{(1)}(T)\subset \sigma_{\f1n}(T_n)$ for all $n$ (where $\sigma_{\f1n}(\cdot)$ denotes the $\f1n$-pseudospectrum).

Next we prove a version of the second resolvent identity for our operator approximation.

\begin{lemma}\label{lemma:strong_on_domain}
	Let $T:\dom(T)\to\h$ be selfadjoint and $T_n$ be defined as in \eqref{eq:Tndef}. Then each $T_n$ is selfadjoint on $\h_n$ and for any $u\in\h$ the implication (ii)$\Rightarrow$(i) holds, where
 	\begin{enumi}
		\item $(T_n-i)^{-1}P_nu\to (T-i)^{-1}u$,
		\item $\big(P_nT-T_nP_n\big)(T-i)^{-1}u\to 0$.
	\end{enumi} 
\end{lemma}
\begin{proof}
	We start by showing that each $T_n$ is selfadjoint. First note that each $T_n$ is automatically bounded, since the $\h_n$ are finite dimensional. Now let $x,y\in\h_n$. Then we have
	\begin{equation}\label{eq:T_n_selfadjoint}
	\begin{split}
		\langle T_n x,y\rangle &= \langle P_n T x,y\rangle
		= \langle Tx,P_n y\rangle
		= \langle Tx , y\rangle\\
		&= \langle x , Ty\rangle
		= \langle P_nx,Ty\rangle
		= \langle x,P_nTy\rangle
		= \langle x,T_ny\rangle.
		\end{split}
	\end{equation}
	and hence $T_n$ is selfadjoint. To prove strong resolvent convergence, we proceed as follows. Since each $T_n$ is selfadjoint, $T_n-i$ is invertible with $\|(T_n-i)^{-1}\|=1$ for all $n$. Now note that on $\h$
	\begin{align*}
		(T_n-i)^{-1}P_n - (T-i)^{-1} &= (T_n-i)^{-1}P_n(T-i)(T-i)^{-1}\\ 
		&\qquad- (T_n-i)^{-1}(T_n-i)P_n(T-i)^{-1} \\
		&\qquad\qquad- (I-P_n)(T-i)^{-1}\\[2mm]
		&= (T_n-i)^{-1}\big(P_n(T-i)-(T_n-i)P_n\big)(T-i)^{-1} \\
		&\qquad- (I-P_n)(T-i)^{-1}\\[2mm]
		&= (T_n-i)^{-1}\big(P_nT-T_nP_n\big)(T-i)^{-1} \\
		&\qquad- (I-P_n)(T-i)^{-1}.
	\end{align*}
	Hence, we have for $u\in\h$
	\begin{align*}
		\left\|\left((T_n-i)^{-1}P_n - (T-i)^{-1}\right)u\right\| &\leq \left\|(T_n-i)^{-1}\right\| \left\|(P_nT-T_nP_n)(T-i)^{-1}u\right\| \\
		&\qquad+ \left\|(I-P_n)(T-i)^{-1}u\right\|\\[2mm]
		&= \left\|(P_nT-T_nP_n)(T-i)^{-1}u\right\| \\
		&\qquad+ \left\|(I-P_n)(T-i)^{-1}u\right\|
	\end{align*}
	The second term on the right hand side converges to 0 since $P_n\to I$ strongly, while the first term converges to 0 by assumption. 

\end{proof}

\begin{prop}\label{lemma:strong_if_T^2}
	If $T$ is selfadjoint, $\bigcup_n\h_n$ form a core of $T$, and $T_n$ is defined as in \eqref{eq:Tndef}, then $T_n\to T$ in strong resolvent sense.
\end{prop}
\begin{proof}
	First note that the condition $(P_nT-T_nP_n)(T-i)^{-1}u\to 0$ from Lemma \ref{lemma:strong_on_domain} is equivalent to $P_nT(I-P_n)(T-i)^{-1}u\to 0$. Now, let $u\in (T-i)\big(\bigcup_n\h_n\big)$ (the image of $\bigcup_n\h_n$ under $T-i$). Then we have
	\begin{align*}
		P_nT(I-P_n)(T-i)^{-1}u\to 0.
	\end{align*}
	Indeed, since $(T-i)^{-1}u\in\bigcup_n\h_n$, we have $(I-P_n)(T-i)^{-1}u=0$ for $n$ large enough, and hence the sequence is identically zero for almost all $n$.
	
	Now from Lemma \ref{lemma:strong_on_domain} we conclude that 
	\begin{align*}
		\left[ (T_n-i)^{-1}P_n - (T-i)^{-1} \right]u\to 0
	\end{align*}
	for all $u\in (T-i)(\bigcup_n\h_n)$. But since $\bigcup_n\h_n$ is a core for $T$, the set $(T-i)\big(\bigcup_n\h_n\big)$ is dense in $\operatorname{ran}(T-i)=\h$. So if $w\in\h$ is arbitrary and $\eps>0$, we can choose $u\in (T-i)(\bigcup_n\h_n)$ such that $\|w-u\|<\eps$ and obtain
	\begin{align*}
		\left\|\left[ (T_n-i)^{-1}P_n - (T-i)^{-1} \right]w\right\| &\leq \left\|\left[ (T_n-i)^{-1}P_n - (T-i)^{-1} \right](w-u)\right\|\\
		&\qquad +\left\|\left[ (T_n-i)^{-1}P_n - (T-i)^{-1} \right]u\right\|\\
		&\leq \left(\left\| (T_n-i)^{-1}\right\|+\left\| (T-i)^{-1} \right\|\right) \|w-u\|\\
		&\qquad +\left\|\left[ (T_n-i)^{-1}P_n - (T-i)^{-1} \right]u\right\|\\
		&\leq 2\eps +\left\|\left[ (T_n-i)^{-1}P_n - (T-i)^{-1} \right]u\right\|\\
		&\to 2\eps
	\end{align*}
	from which we conclude that $ (T_n-i)^{-1}P_n \to (T-i)^{-1}$ strongly on $\h$.
%	\begin{align*}
%		\|P_nT(I-P_n)(T-i)^{-1}u\|^2 &= \left\langle P_nT(I-P_n)(T-i)^{-1}u\,,\,P_nT(I-P_n)(T-i)^{-1}u \right\rangle\\
%		&= \left\langle TP_nT(I-P_n)(T-i)^{-1}u\,,\,(I-P_n)(T-i)^{-1}u \right\rangle
%	\end{align*}
%	\todo{This does not work as planned! Fix this.}
%	(note that since $u\in\dom(T)$, we have $(T-i)^{-1}u\in\dom(T^2)$). Since $(I-P_n)\xrightarrow{s} 0$ as $n\to\infty$, it follows immediately that $\|P_nT(I-P_n)(T-i)^{-1}u\|\to 0$ as well. We conclude from Lemma \ref{lemma:strong_on_domain} that 
%	\begin{align*}
%		(T_n-i)^{-1}P_nu\to (T-i)^{-1}u\qquad\forall u\in\dom(T).
%	\end{align*}
%	It remains to remove the restriction $u\in\dom(T)$. To this end, let $v\in\h$ be arbitrary. Given $\eps>0$ choose $u_\eps\in\dom(T)$ such that $\|v-u_\eps\|<\eps$. Then compute
%	\begin{align*}
%		\left\|\big((T_n-i)^{-1}P_n- (T-i)^{-1}\big)v\right\| &\leq \left\|\big((T_n-i)^{-1}P_n- (T-i)^{-1}\big)u_\eps\right\|\\
%		&\qquad + \left\|\big((T_n-i)^{-1}P_n- (T-i)^{-1}\big)(v-u_\eps)\right\|\\
%		&\leq \left\|\big((T_n-i)^{-1}P_n- (T-i)^{-1}\big)u_\eps\right\|\\
%		&\qquad + \left(\left\|(T_n-i)^{-1}\right\|+ \left\|(T-i)^{-1}\right\|\right) \|v-u_\eps\|\\
%		&\leq \left\|\big((T_n-i)^{-1}P_n- (T-i)^{-1}\big)u_\eps\right\|+2\eps
%	\end{align*}
%	Taking $n\to\infty$, we obtain
%	\begin{align*}
%		\limsup_{n\to\infty}\left\|\big((T_n-i)^{-1}P_n- (T-i)^{-1}\big)v\right\| &\leq 2\eps.
%	\end{align*}
%	Since $\eps>0$ was arbitrary, the result follows.
\end{proof}

The following definitions from \cite{BMC}, which are related to the essential spectrum, will be used frequently in the sequel. The \emph{limiting essential spectrum:}
\begin{align*}
	\limess &:= \{\lambda\in \C\,|\, \exists x_k\in\dom(T_{n_k}) : \|x_k\|=1\,\forall k,\, x_k\rightharpoonup 0,\, \|(T_{n_k}-\lambda)x_k\|\to 0\},
\end{align*}
the limiting $\eps$-near spectrum:
\begin{align*}
	\limpseu{\eps} &:= \{\lambda\in \C\,|\, \exists x_k\in\dom(T_{n_k}) : \|x_k\|=1\,\forall k,\, x_k\rightharpoonup 0,\, \|(T_{n_k}-\lambda)x_k\|\to \eps\},
\end{align*}
the essential numerical range
\begin{align*}
	W_e(T) &:= \{\lambda\in \C\,|\, \exists x_k\in\dom(T) : \|x_k\|=1\,\forall k,\, x_k\rightharpoonup 0,\, \langle Tx_k,x_k\rangle \to \lambda \}
\end{align*}
and the limiting essential numerical range
\begin{align*}
	W_e\Tn &:= \{\lambda\in \C\,|\, \exists x_k\in\dom(T_{n_k}) : \|x_k\|=1\,\forall k,\, x_k\rightharpoonup 0,\, \langle T_{n_k}x_k,x_k\rangle \to \lambda \}.
\end{align*}
The essential limiting spectrum was originally introduced in \cite{BBL} in the context of Galerkin approximation and later adapted to a more general framework in \cite{B18}, where the set $\limpseu{\eps}$ was introduced. The essential numerical range was originally introduced by Stampfli and Williams in \cite{SW} for bounded operators and recently extended to unbounded operators in \cite{BMC}. It was shown there that the essential numerical range is a convenient tool when studying spectral and pseudospectral pollution of operator approximations. This fact will prove very useful to our purpose as we shall see in the following.

\begin{lemma}\label{lem:LambdaSigma}
	\begin{enumi}
	\item For any closed, densely defined operator $T$ on $\h$ one has
	\begin{align*}
		\bigcap_{\eps>0}	\bigcup_{\delta\in(0,\eps]} \limpseu{\delta}\;\subset\; \limess.
	\end{align*}
	\item The above inclusion holds, if $\bigcap_{\eps>0}	\bigcup_{\delta\in(0,\eps]}$ is replaced by $\bigcap_{k}	\bigcup_{\delta\in(0,\eps_k]}$ for any sequence $(\eps_k)$ with $\eps_k\to 0$.
	\end{enumi}
\end{lemma}
\begin{proof}
	We first prove (i). Let $\lambda\in \bigcap_{\eps>0}\bigcup_{\delta\in(0,\eps]} \limpseu{\delta}$. Then for all $\eps>0$ there exists $\delta\in(0,\eps]$ and a sequence $(x_k)$ with $x_k\in\dom(T_{n_k})$ (for some subsequence $(n_k)$) such that
	\begin{itemize}
		\item $\|x_k\|=1$ for all $k$
		\item $x_k\rightharpoonup 0$ as $k\to\infty$
		\item $\|(T_{n_k}-\lambda)x_k\|\to\delta$.
	\end{itemize}
	Hence, for every $m\in\N$ there exists a sequence $(x_k^{(m)})_{k\in\N}$ with $\bigl\|x_k^{(m)}\bigr\|=1$, $x_k^{(m)}\xrightharpoonup{k\to\infty}~0$ and 
	\begin{equation*}
		\lim_{k\to\infty} \left\|(T_{n_k(m)}-\lambda)x_k^{(m)}\right\| <\f1m.
	\end{equation*}
	The notation $n_k(m)$ indicates that the corresponding subsequence of $(T_n)$ depends on $m$. Now, construct a diagonal sequence as follows. Since $\h$ is separable, the weak topology is metrisable on the unit ball. Let $d$ denote a corresponding metric. Now, for any given $m\in\N$, choose $k_m\in\N$ large enough such that
	\begin{align*}
		d\big(x_{k_m}^{(m)},0\big)&<\f1m\\
		\left\|(T_{n_{k_m}(m)}-\lambda)x_{k_m}^{(m)}\right\| &<\f1m.
	\end{align*}
	Then for the sequence $y_m:=x^{(m)}_{k_m}$, one has $\|y_m\|=1$ for all $m$, $d(y,0)\to 0$ and $\|(T_{n_{k_m}(m)}-\lambda)y_m\|\to 0$ as $m\to\infty$. Hence $\lambda\in\limess$.
	
	The proof of (ii) is now trivial, since the sequence of sets $\bigcup_{\delta\in(0,\eps]} \limpseu{\delta}$ is shrinking with $\eps$.
\end{proof}

Next, we prove convergence of the algorithm $\Gamma_n^{(1)}$. By the conditions in \eqref{eq:Omega} and Lemma \ref{lemma:strong_on_domain} and Proposition \ref{lemma:strong_if_T^2}, we have $T_n\xrightarrow{s}T$ for all $T\in\Om$. Let $\eps>0$. Then for large enough $n\in\N$ one has $\f1n<\eps$ and thus
\begin{equation}
	\Gamma_n^{(1)}(T)\subset\sigma_{\f1n}(T_n)\subset\sigma_\eps(T_n).
\end{equation}
According to \cite[Th. 3.6 ii)]{B18}, pseudospectral pollution of the approximation $T_n\to T$ is confined to
\begin{align*}
	\sigma_e\Tn\cup\bigcup_{\delta\in(0,\eps]}\limpseu{\delta}.
\end{align*}
Hence, if $\lambda_n\in\Gamma_n^{(1)}(T)$ and $\lambda_n\to\lambda\in\C$, it follows that
\begin{align*}
	\lambda &\in  \sigma_\eps(T)\cup \limess \cup\bigcup_{\delta\in(0,\eps]}\limpseu{\delta}.
\end{align*}
 Since this holds for any $\eps>0$, we immediately obtain
 \begin{align}\label{eq:liminc}
 	\lambda\in \bigcap_{\eps>0}\left[ \sigma_\eps(T)\cup \limess \cup\bigcup_{\delta\in(0,\eps]}\limpseu{\delta}\right].
 \end{align}
\begin{lemma}\label{lem:followsfrom4}
	It follows from \eqref{eq:liminc} that 
	\begin{align*}
		\lambda\in \sigma(T)\cup\limess.
	\end{align*}
\end{lemma}
\begin{proof}
	Let \eqref{eq:liminc} hold. Then 
	\begin{itemize}
		\item[-] Either there exists $\eps_0>0$ such that $\lambda\in \sigma_\eps(T)\cup\limess$ for all $\eps\in(0,\eps_0)$, or
		\item[-] there exists a sequence $\eps_k$ with $\eps_k\searrow0$ and $\lambda\in \bigcup_{\delta\in(0,\eps_k]}\limpseu{\delta}$ for all $k$.
	\end{itemize}
	In the first case, it follows that 
	\begin{align*}
		\lambda &\in \bigcap_{\eps>0}\Big(\sigma_\eps(T)\cup\limess\Big) \\
		&= \left(\bigcap_{\eps>0}\sigma_\eps(T)\right)\cup\limess \\
		&= \sigma(T)\cup\limess.
	\end{align*}
	In the second case, we have 
	\begin{align*}
		\lambda &\in \bigcap_{k\in\N}\bigcup_{\delta\in(0,\eps_k]}\limpseu{\delta}\\
		&\subset \limess,
	\end{align*}
	by Lemma \ref{lem:LambdaSigma} (ii).
\end{proof}
To conclude, we apply \cite[Th. 6.1]{BMC} to show that spectral pollution is in fact absent for $T\in\Omega_1$. Indeed, let $\lambda_n\in\Gamma_n^{(1)}(T)$ with $\lambda_n\to\lambda$. Then by Lemma \ref{lem:followsfrom4} and \cite[Prop. 5.6, Th. 6.1]{BMC}, we get
\begin{align*}
	\lambda &\in \sigma(T)\cup\limess\\
	&\subset\sigma(T)\cup W_e(T)\\
	&= \sigma(T)\cup\operatorname{conv}(\widehat\sigma_e(T))\setminus\{\pm\infty\}\\
	&= \sigma(T)\cup \sigma_e(T)\\
	&= \sigma(T).
\end{align*}
It remains to prove spectral inclusion, i.e. nothing is missed by $\Gamma_n^{(1)}(T)$.

\begin{lemma}
	For every $\lambda\in\sigma(T)$ there exist $\lambda_n\in\Gamma_n^{(1)}(T)$ such that $\lambda_n\to\lambda$.
\end{lemma}
\begin{proof}
	Let $\lambda\in\sigma(T)$. A simple adaption of the proof of \cite[Th. VIII.24]{RS} shows that there exists a sequence $(\mu_n)$ with $\mu_n\in\sigma(T_n)$ and $\mu_n\to\lambda$. 
	
	For each $n$, there exists $\lambda_n\in G_n$ such that $|\mu_n-\lambda_n|<\f{1}{n}$ and hence $\|(T_n-\lambda_n)^{-1}\|_{L(\h_n)}\geq n$ which implies $\lambda_n\in\Gamma_n^{(1)}(T)$. Since $|\mu_n-\lambda_n|\to 0$ and $\mu_n\to\lambda$, it follows that $\lambda_n\to\lambda$.
\end{proof}

\paragraph{Conclusion.}
We have shown that 
\begin{enuma}
	\item If $\lambda_n\in\Gamma_n^{(1)}(T)$ and $\lambda_n\to\lambda$, then $\lambda\in\sigma(T)$.
	\item If $\lambda\in\sigma(T)$, then there exist $\lambda_n\in\Gamma_n^{(1)}(T)$ with $\lambda_n\to\lambda$.
\end{enuma}
It only remains to show that this implies Attouch-Wets convergence. We recall that $d_{\mathrm{AW}}(X_n,X)\to 0$ if and only if $d_K(X_n,X)\to 0$ for all $K\subset \C$ compact, where
\begin{align*}
	d_K(X,Y) := \max\left\{ \sup_{x\in X\cap K}\dist(x,Y)\,,\,\sup_{y\in Y\cap K}\dist(y,X) \right\}.
\end{align*}
\begin{prop}\label{prop:AW_convergence}
	If (a), (b) above hold, then $d_{\mathrm{AW}}(\Gamma_n^{(1)}(T),\sigma(T))\to 0$.
\end{prop}
\begin{proof}
	Let $K\subset\C$ be compact. We will show that if (a), (b) hold, then both distances $\sup_{z\in \Gamma_n^{(1)}(T)\cap K}\dist(z,\sigma(T))$ and $\sup_{w\in \sigma(T)\cap K}\dist\big(w,\Gamma_n^{(1)}(T)\big)$ converge to zero. We begin with the latter.
	
	Let $\eps>0$. For all $w\in\sigma(T)\cap K$, the ball $B_\eps(w)$ contains infinitely many elements $z_n\in\Gamma_n^{(1)}(T)$, by (b). The collection $\{B_\eps(w)\,|\,w\in \sigma(T)\cap K\}$ forms an open cover of the compact set $\sigma(T)\cap K$. Hence, there exist finitely many $w_1,\dots ,w_k\in \sigma(T)\cap K$ such that $B_\eps(w_1),\dots,B_\eps(w_k)$ cover $\sigma(T)\cap K$. Now, any $w\in \sigma(T)\cap K$ is contained in some $B_\eps(w_i)$ and hence $\dist(w,\Gamma_n^{(1)}(T))<\eps$ for any $w\in\sigma(T)\cap K$, as soon as $n=n(i)$ is large enough. But since there are only finitely many $B_\eps(w_i)$, one will have $\dist(w,\Gamma_{n_0}^{(1)}(T))<2\eps$ for all $w\in\sigma(T)\cap K$ for $n_0 = \max\{n_i\,|\,i=1,\dots,k\}$.
	
	To show that $\sup_{z\in \Gamma_n^{(1)}(T)\cap K}\dist(z,\sigma(T))\to 0$ as $n\to\infty$, note that since all sets $\Gamma_n^{(1)}(T)\cap K$ are compact, we can choose a sequence $z_n\in\Gamma_n^{(1)}(T)\cap K$ such that 
	\begin{align*}
		\sup_{z\in \Gamma_n^{(1)}(T)\cap K}\dist(z,\sigma(T)) = \dist(z_n,\sigma(T)).
	\end{align*}
	Since the sequence $(z_n)$ is obviously bounded, we can extract a convergent subsequence $z_{n_j}\to z_0\in K$. Now use assertion (a) from above to conclude that in fact $z_0\in\sigma(T)\cap K$. This readily implies 
	\begin{align*}
		\sup_{z\in \Gamma_{n_j}^{(1)}(T)\cap K}\dist(z,\sigma(T)) &= \dist(z_{n_j},\sigma(T))\to 0.
	\end{align*}
	Since the same reasoning can be applied to every subsequence of the sequence $$\left(\sup_{z\in \Gamma_n^{(1)}(T)\cap K}\dist(z,\sigma(T))\right)_{n\in\N},$$ we conclude that the whole sequence converges to zero.
\end{proof}

\section{Relatively Compact Perturbations}\label{sec:Perturbations}

In this section we show that Theorem \ref{th:mainth} remains true for certain relatively compact, bounded perturbations of selfadjoint operators. More precisely, we have
\begin{theorem}\label{th:perturbation}
	Define a computational problem by
	\begin{align*}
		\Omega_2&:=\left\{H=T+V:\dom(T)\to\h \;\Bigg|\; \parbox{7.7cm}{$T \text{ \rm selfadjoint, semibounded, }\bigcup_{n\in\N}\h_n\text{ \rm core for } T,\\
		\sigma(T)=\sigma_{e5}(T), \;\;\widehat\sigma_e(T)\text{ \rm convex, }\\
	 		V\in L(\h)\text{ \rm and } V,V^*\text{ \rm are } $T$\text{\rm-compact.}
	 		$} \right\}
	 \end{align*}
	(where $\sigma_{e5}(T)$ will be defined below). For every $H\in\Om_2$, choose a decomposition $H=T+V$ as in the definition of $\Om_2$ and define the maps $s_T(H):=T$ and $s_V(H):=V$. Then let
	 \begin{align*}
	 		\Lambda_2 &:= \left\{ f_{i,j,n}\circ s_T\,|\,1\leq i,j\leq n,\;n\in\N \right\}\cup \left\{ f_{i,j,n}\circ s_V\,|\,1\leq i,j\leq n,\;n\in\N \right\},
	\end{align*}
	where $f_{i,j,n}$ are the evaluation functions producing the $(i,j)$th matrix elements (see \eqref{eq:def_Lambda_1}).
	Then one has $\SCI(\Om_2,\Lambda_2,\sigma(\cdot))=1$.
\end{theorem}
\begin{remark}
	Note that the information provided to the algorithm in $\Lambda_2$ includes the decomposition of $H\in\Om_2$ into a selfadjoint part $T$ and a perturbation $V$. This means, that the algorithm \emph{does not have to compute this decomposition}. It gets it for free. This is a reasonable assumption in many applications as we will see in Section \ref{sec:Schroedinger}.
\end{remark}

Note the additional assumption $\sigma(T)=\sigma_{e5}(T)$ in the selfadjoint part $T$. This will be needed later in order to exclude spectral pollution of the algorithm. 

\subsection{Proof of Theorem \ref{th:perturbation}}

\paragraph{Specrtum of $\boldsymbol H$.} 
The proof of Theorem \ref{th:perturbation} is via perturbation theory. We first focus on the spectrum of an operator $H\in\Om_2$ and recall some of the varying definitions of essential spectra. In Section \ref{sec:selfadjoint}, we introduced the set $\sigma_{e2}$ as 
\begin{align*}
	\sigma_{e2}(H):=\{\lambda\in \C\,|\, \exists x_k\in\dom(H) : \|x_k\|=1\,\forall k,\, x_k\rightharpoonup 0,\, \|(T-\lambda)x_k\|\to 0\}.
\end{align*}
We will need another version of the essential spectrum, which is sometimes denoted $\sigma_{e5}$:
\begin{align*}
	\sigma_{e5}(H):=\C\setminus\Delta_5(H),
\end{align*}
where $\Delta_5(H)$ denotes the union of all components of the set $\{\lambda\in\C\,|\,H-\lambda\text{ is semi-Fredholm}\}$ which intersect $\rho(H)$. For more details, cf. \cite[Ch. IX]{EE}. The following results are classical.

\begin{theorem}[{\cite[Th. IX.1.5]{EE}}]\label{th:EE}
	For any closed, densely defined operator $H$ on $\h$, one has $\lambda\notin\sigma_{e5}(H)$ if and only if $T-\lambda$ is Fredholm with $\operatorname{ind}(T-\lambda)=0$ and a deleted neighbourhood of of $\lambda$ lies in $\rho(H)$.
\end{theorem}
In other words, if $\lambda\notin\sigma_{e5}(H)$, then $\lambda$ is an isolated eigenvalue of finite multiplicity. Furthermore, the following perturbation result is known.

\begin{theorem}[{\cite[XIII.4, Cor. 2]{RS}}]\label{th:RS}
	Let $T$ be a selfadjoint operator on $\h$ and $V$ relatively compact w.r.t. $T$. Then
	\begin{enumi}
		\item $H:=T+V$ is closed on $\dom(T)$ and
		\item $\sigma_{e5}(H)=\sigma_{e5}(T)$.
	\end{enumi}
\end{theorem}
From Theorems \ref{th:EE} and \ref{th:RS} we immediately conclude that for all $H\in\Om_2$ the spectrum of $\h$ is of the form 
\begin{align*}
	\sigma(H) = \sigma(T)\cup\{\lambda_1, \lambda_2,\dots\},
\end{align*}
with isolated eigenvalues $\lambda_i\in\C$.

\paragraph{Strong resolvent convergence.}
Let $P_n:\h\to\h_n$ be defined as in Section \ref{sec:selfadjoint} and set $V_n:=P_nV|_{\h_n}$. 
\begin{lemma}\label{lemma:V_n_convergence}
	For $V_n$ defined as above, we have the following
	\begin{enumi}
		\item $(V_n)^* = (V^*)_n$ (i.e. compression to $\h_n$ commutes with taking the adjoint) and
		\item $V_nP_n\to V$ strongly in $\h$.
		\item $V_n^*P_n\to V^*$ strongly in $\h$.
	\end{enumi}
\end{lemma}
\begin{proof}
	Assertion (i) is easily shown by an analogous calculation to \eqref{eq:T_n_selfadjoint}. 
	
	To see assertion (ii), let $u\in\h$ and note that then $P_nu\to u$ strongly. By continuity of $V$, it immediately follows that $VP_nu\to Vu$ in $\h$. Hence, from the definition of $V_n$ we conclude that 
	\begin{align*}
		V_nP_nu &= P_nV|_{\h_n}P_nu = \underbrace{P_n}_{\to I\text{ strongly}}\underbrace{VP_nu}_{\to Vu} \,\to\, Vu.
	\end{align*}
	Assertion (iii) now immediately follows by combining (i) and (ii).
\end{proof}
The next lemma shows that even the perturbed operators $H_n$ converge in strong resolvent sense.

\begin{lemma}\label{lemma:H_conv_strongly}
	For $H\in\Om_2$ and $H_n=P_nH|_{\h_n}$, one has $H_n\to H$ and $H_n^*\to H^*$ in strong resolvent sense.
\end{lemma}
\begin{proof}
	Let us first note that for $H=T+V$ as in the definition of $\Om_2$ and $z\in\rho(H_n)\cap\rho(H)$ we have the following decomposition of the resolvents
	\begin{align}
		(H\!-\!z)^{-1}\!\! -\! (H_n\!-\!z)^{-1}\!P_n &= (T\!-\!z)^{-1}\!\left[ I\!+\!V(T\!-\!z)^{-1} \right]^{-1}\!\!\! -\! (T_n\!-\!z)\!\left[ I\!+\!V_n(T_n\!-\!z)^{-1} \right]^{-1}\!\!P_n \nonumber \\[1mm]\nonumber
		&= (T_n\!-\!z)^{-1}\!\Big[ P_n\!\left( I\!+\!V(T\!-\!z)^{-1} \right)^{-1} \!\!\! -\!\left( I\!+\!V_n(T_n\!-\!z)^{-1} \right)^{-1}\!P_n \Big]\\
		& \quad + \Big[ (T_n-z)^{-1}P_n - (T-z)^{-1} \Big]\!\left(I+V(T-z)^{-1}\right)^{-1}\,. \label{eq:H_decomposition}
	\end{align}
	Note that since $V$ is bounded and $T$ is selfadjoint, it is easy to find $z\in\rho(H)$ such that $z\in\rho(H_n)$ for all $n$ by choosing $\im(z)$ large enough. In fact, a standard Neumann series argument shows that for $|\im(z)|>1+\|V\|$ one has
	\begin{equation}\label{eq:uniform_resolvent_bound}
	\begin{split}	
		\left\|(H_n-z)^{-1}\right\|	 &\leq \left\|(T_n-z)^{-1}\left(I+V_n(T_n-z)^{-1}\right)^{-1}\right\|\\
		&\leq \left(\left\|(T_n-z)^{-1}\right\|^{-1}-\|V_n\|\right)^{-1}\\
		&\leq \left(|\im(z)|-\|V\|\right)^{-1}\\
		&\leq 1
	\end{split}
	\end{equation}
	
	For such $z$, in order to estimate $(H-z)^{-1} - (H_n-z)^{-1}P_n$, we will estimate each term on the right hand side of \eqref{eq:H_decomposition} in turn. We start with the second term. For arbitrary $u\in\h$ the term
	\begin{align*}
		\left\|\Big[ (T_n-z)^{-1}P_n - (T-z)^{-1} \Big]\!\left(I+V(T-z)^{-1}\right)^{-1}u\right\|
	\end{align*}
	goes to 0 as $n\to\infty$, since $(T_n-z)^{-1}P_n \to (T-z)^{-1}$ strongly by Proposition \ref{lemma:strong_if_T^2}.
	
	In order to treat the first term on the right hand side of \eqref{eq:H_decomposition}, we use the second resolvent identity. In fact, the term $(T_n-z)^{-1}\big[P_n(I+V(T-z)^{-1})^{-1}-(I+V_n(T_n-z)^{-1} )^{-1}P_n\big]$ is equal to
	\begin{align*}
		(T_n\!-\!z)^{-1}\!\left(I+V_n(T_n\!-\!z)^{-1}\right)^{-1}\!\left[V_n(T_n\!-\!z)^{-1}P_n - P_nV(T\!-\!z)^{-1}\right]\!\left(I+V(T\!-\!z)^{-1}\right)^{-1}.
	\end{align*}
	The norm of this operator, applied to $u\in\h$ is controlled by
	\begin{align*}
		\left\|(T_n\!-\!z)^{-1}\!\left(I+V_n(T_n\!-\!z)^{-1}\right)^{-1}\right\|\left\|\left[V_n(T_n\!-\!z)^{-1}P_n - P_nV(T\!-\!z)^{-1}\right]\!\left(I+V(T\!-\!z)^{-1}\right)^{-1}u\right\|.
	\end{align*}
	By our choice of $z$, the first factor is less than 1 (see eq. \eqref{eq:uniform_resolvent_bound}). Setting $v:=\left(I+V(T-z)^{-1}\right)^{-1}u$ to simplify notation, it remains to estimate
	\begin{align*}
		\left\|\left[V_n(T_n-z)^{-1}P_n - P_nV(T-z)^{-1}\right]v\right\|.
	\end{align*}
	But this clearly converges to 0 as $n\to\infty$, since $V_nP_n\to V$ and $(T_n-z)^{-1}P_n \to (T-z)^{-1}$ strongly by Proposition \ref{lemma:strong_if_T^2} and Lemma \ref{lemma:V_n_convergence} (ii).
	
	Finally, we note that the strong convergence of $H_n^*$ follows immediately from the above calculations and Lemma \ref{lemma:V_n_convergence}.
\end{proof}

\paragraph{The algorithm.}
The algorithm for $\Om_2,\Lambda_2$ is defined almost identically to that in Section \ref{sec:selfadjoint}. Namely, we define
\begin{align}\label{eq:Gamma2}
	\Gamma_n^{(2)}(H):=\left\{ \lambda\in G_n\,\bigg|\, \min\left\{s(H_n-\lambda),s(H_n^*-\overline\lambda)\right\}\leq\f1n \right\}\cup\Gamma_n^{(1)}(T).
\end{align}
Note that we have $\min\{s(M-\lambda),s(M^*-\overline\lambda)\}=\|(M-\lambda)^{-1}\|^{-1}$ for any $n\times n$ matrix $M$ (cf. \cite{Hansen11}). Since we have already shown that $\Gamma_n^{(1)}$ approximates $\sigma(T)$ correctly and that $\sigma(T)=\sigma_{e5}(T)=\sigma_{e5}(H)$, we know that $\Gamma_n^{(2)}$ will not miss anything in $\sigma_{e5}(H)$. Thus, it only remains to prove absence of spectral pollution and spectral inclusion for the discrete set $\sigma(H)\setminus\sigma_{e5}(H)$ for the algorithm 
	$$\tilde\Gamma_n(H):=\left\{ \lambda\in G_n\,\bigg|\, \min\left\{s(H_n-\lambda),s(H_n^*-\overline\lambda)\right\}\leq\f1n \right\}$$ 
This will be done in the remainder of this section. 

However, let us first take a moment to assure that $\Gamma_n^{(2)}$ defines a reasonable algorithm. Clearly, each $\Gamma_n^{(2)}$ depends only on $\big\langle Te_i^{(n)},e_j^{(n)}\big\rangle$ and $\big\langle Ve_i^{(n)},e_j^{(n)}\big\rangle$, $1\leq i,j\leq k_n$. Moreover, by Lemma \ref{lemma:finite_testing} it only requires finitely many algebraic operations on these numbers to determine whether $\lambda\in G_n$ belongs to the set $\left\{ \lambda\,|\, \min\left\{s(H_n-\lambda),s(H_n^*-\overline\lambda)\right\}\leq\f1n \right\}$. Finally, since $\Lambda_2$ contains all matrix elements $\big\langle Te_i^{(n)},e_j^{(n)}\big\rangle$, it follows from the comments made in Section \ref{sec:selfadjoint} that $\Gamma_n^{(1)}$ is an admissible algorithm as well. 
\begin{remark}\label{Remark}
	We note that the choice $\f1n$ as an upper bound for $s(H_n-\lambda)$ in \eqref{eq:Gamma2} is arbitrary. The proof below will show that one could equally well have chosen 
	\begin{align*}
		\Xi_n(H)&:=\left\{ \lambda\in G_n\,\bigg|\, \min\left\{s(H_n-\lambda),s(H_n^*-\overline\lambda)\right\}\leq\f3n \right\}\cup\Gamma_n^{(1)}(T)
	\end{align*}
	instead of $\Gamma^{(2)}_n(H)$. This fact will be used in Section \ref{sec:Schroedinger}.
\end{remark}

\paragraph{Spectral pollution.}
Let us prove that the approximation $\Gamma_n^{(2)}(H)$ does not have spectral pollution for $H\in\Om_2$. To this end, note that again $\tilde \Gamma_n(H)\subset\sigma_\eps(H_n)$ for $\eps>0$ fixed and $n$ large enough. 
According to \cite[Th. 3.6 ii)]{B18}, $\eps$-pseudospectral pollution of the approximation $H_n\to H$ is confined to
\begin{align*}
	\sigma_e\big( (H_n)_{n\in\N} \big)\cup\sigma_e\big( (H_n^*)_{n\in\N} \big)^*\cup\bigcup_{\delta\in(0,\eps]}\Lambda_{e,\delta}\big( (H_n)_{n\in\N} \big).
\end{align*}
Hence, for any sequence $\lambda_n\in\tilde\Gamma_n(H)$ with $\lambda_n\to\lambda\in\C$ we have 
\begin{align*}
	\lambda\in \bigcap_{\eps>0}\left(\sigma_\eps(H)\cup\sigma_e(H_n)_{n\in\N}\cup\sigma_e\big( (H_n^*)_{n\in\N} \big)^*\cup\bigcup_{\delta\in(0,\eps]}\Lambda_{e,\delta}\big( (H_n)_{n\in\N} \big)\right).
\end{align*}
We conclude from Lemmas \ref{lem:LambdaSigma} and \ref{lem:followsfrom4} that in fact
\begin{align*}
	\lambda\in \sigma(H)\cup\sigma_e\big( (H_n)_{n\in\N} \big)\cup\sigma_e\big( (H_n^*)_{n\in\N} \big)^*.
\end{align*}
Furthermore, by \cite[Th. 6.1]{BMC} we have $\sigma_e\big( (H_n)_{n\in\N} \big)\cup\sigma_e\big( (H_n^*)_{n\in\N} \big)^*\subset W_e(H)$ and hence $\lambda\in \sigma(H)\cup W_e(H)$. In order to exclude spectral pollution it only remains to prove $W_e(H)\subset \sigma(H)$.
\begin{lemma}\label{lemms:W_in_sigma}
	For $H=T+V\in\Om_2$ one has $W_e(H)\subset\sigma_e(H)$.
\end{lemma}
\begin{proof}
	Let $H=T+V$ with $T$ selfadjoint, semibounded and $V\in L(\h)$ such that $V,\,V^*$ are $T$-compact. Then denoting $\re(V):=\f12(V+V^*)$ and $\im(V):=\f{1}{2i}(V-V^*)$ we have that $V=\re(V)+i\im(V)$ with $\re(V),\,\im(V)$ relatively compact w.r.t. $T$. 
	Applying \cite[Th. 4.5]{BMC} we conclude that $W_e(H)=W_e(T)$.

But now by our assumptions on $T$, we can see from \cite[Th. 3.8]{BMC} that 
$$W_e(T)=\operatorname{conv}\big(\widehat\sigma_e(T)\big)\setminus\{\pm\infty\}\subset\sigma_e(T) = \sigma_e(H).$$
\end{proof}
Overall we have shown that for any sequence $\lambda_n\in\tilde\Gamma_n(H)$ which converges to some $\lambda\in\C$ we necessarily have $\lambda\in\sigma(H)$, in other words, spectral pollution does not exist.

\paragraph{Spectral inclusion.}
It remains to show that the approximation $(\Gamma_n^{(2)}(H))$ is spectrally inclusive, i.e. that for any $\lambda\in\sigma(H)$ there exists a sequence $\lambda_n\in\Gamma_n^{(2)}(H)$ such that $\lambda_n\to\lambda$. As explained above, the existence of such a sequence is already guaranteed for all $\lambda\in\sigma_{e5}(H)$. 
\begin{lemma}
	For every $\lambda\in\sigma(H)\setminus\sigma_{e5}(H)$ there exists a sequence $\lambda_n\in\tilde\Gamma(H)$ with $\lambda_n\to\lambda$.
\end{lemma}
\begin{proof}
	First note that by Theorem \ref{th:EE} $\lambda$ is an isolated point. Moreover, we have seen in the proof of Lemma \ref{lemms:W_in_sigma} that $\sigma_e\big( (H_n)_{n\in\N} \big)\cup\sigma_e\big( (H_n^*)_{n\in\N} \big)^*\subset\sigma_e(H)$ and hence $\lambda$ does not belong to this set either. From Lemma \ref{lemma:H_conv_strongly} and \cite[Th. 2.3 i)]{B18} we conclude that there exists a sequence $\mu_n\in\sigma(H_n)$ with $\mu_n\to\lambda$. 
	
	Now, by definition of $G_n$, for each $n$ there exists $\lambda_n\in G_n$ such that $|\mu_n-\lambda_n|<\f1n$ and hence $\|(H_n-\lambda_n)^{-1}\|_{L(\h_n)}\geq n$ which implies $\lambda_n\in\tilde\Gamma_n(H)$. Since $|\mu_n-\lambda_n|\to 0$ and $\mu_n\to\lambda$, it follows that $\lambda_n\to\lambda$. 
\end{proof}

\paragraph{Conclusion.}
Overall we have shown that 
\begin{itemize}
	\item[(a$'$)] If $\lambda_n\in\Gamma_n^{(2)}(H)$ and $\lambda_n\to\lambda$, then $\lambda\in\sigma(H)$.
	\item[(b$'$)] If $\lambda\in\sigma(H)$, then there exist $\lambda_n\in\Gamma_n^{(2)}(H)$ with $\lambda_n\to\lambda$.
\end{itemize}
As in Proposition \ref{prop:AW_convergence} this implies $d_{\mathrm{AW}}\big(\Gamma_n^{(2)}(H),\sigma(H)\big)\to 0$.

\section{Application to Schr\"odinger Operators}\label{sec:Schroedinger}

In this section we will apply the results of Sections \ref{sec:selfadjoint} and \ref{sec:Perturbations} to Schr\"odinger operators on $L^2(\R^d)$. More specifically, for fixed $M>0$ and $C\subset\R^d$ compact we define
\begin{align}\label{eq:def_Omega3}
	\Om_3 &:= \left\{ -\Delta+V\,\big|\,V\in \mathcal C^1(\R,\C),\, \supp(V)\subset C,\, \|V\|_{\mathcal C^1}\leq M\right\},
\end{align}
where $\|V\|_{\mathcal C^1} = \|V\|_\infty+\|\nabla V\|_\infty$. Since in the above definition the potential functions $V$ are compactly supported and bounded by $M$, every $H\in\Om_3$ is a relatively compact perturbation of the free Laplacian with domain $H^2(\R^d)$. In fact, our assumptions on $V$ have been chosen such that every $H\in\Om_3$ even satisfies all conditions formulated in the set $\Om_2$ in Theorem \ref{th:perturbation}.

In order to define the computational problem, we choose a finite lattice in $\R^d$
\begin{align*}
	L_n := \left\{\f in \;\Big|\;i\in \mathbb Z^d, |i|<n \right\}.
\end{align*} 
Moreover, let $\h_n$ denote the subspace of $L^2(\R^d)$ spanned by all characteristic  functions of cubes of edge length $\f1n$ with centres inside a ball of radius $n$:
\begin{align*}
	\widehat\h_n:=\operatorname{span}\left\{ \chi_{i+[0,\f1n)^d}\,\Big|\,i\in L_n \right\}
\end{align*}
It is easily seen by smooth approximation that $P_{\widehat\h_n}\to I$ strongly in $L^2(\R^d)$. However, none of the basis functions $\chi_{i+[0,\f1n)^d}$ are contained in the domain of $-\Delta$. In order to circumvent this, the space we will actually work with will be
\begin{align}\label{eq:Hn_Def}
	\h_n &:=\operatorname{span}\left\{ \widehat\chi_{i+[0,\f1n)^d}\,\Big|\, i\in L_n \right\},
\end{align}
where the hat denotes the Fourier transform in $L^2(\R^d)$. For any enumeration $i_k$ of the set $L_n$, we define
\begin{align*}
	e_k^{(n)}:=n^{\f d2}\cdot\widehat\chi_{i_k+[0,\f1n)^d},
\end{align*}
where the normalisation constant $n^{\f d2}$ is chosen such that $\big\|e_k^{(n)}\big\|_{L^2(\R^d)}=1$ for all $n\in\N$.
These are smooth functions in $L^2(\R^d)$ and it is easily checked that their first and second derivatives are again in $L^2(\R^d)$. We note that the functions $e_k^{(n)}$ can be calculated explicitly. Indeed, one has
\begin{align*}
	e_k^{(n)}(\xi) = \left(\f{n}{2\pi}\right)^{\f d2}\prod_{j=1}^d\f{e^{\text{i}\xi_j((i_k)_j+\f1n)} - e^{\text{i}\xi_j (i_k)_j}}{\xi_j},
\end{align*}
where $(i_k)_j$ denotes the $j$'th component of the vector $i_j$ and $\xi=(\xi_1,\dots,\xi_d)\in\R^d$.
Using this explicit representation, it can be easily seen that we have the following.
\begin{lemma}\label{lemma:e_k-Bound}
	For each $n\in\N$ one has  
	\begin{align*}
		\big\|e_k^{(n)}\big\|_\infty, \,\big\|\nabla e_k^{(n)}\big\|_\infty\leq (2\pi)^{-\f d2}dn^{3-\f{d}{2}}
	\end{align*}
	for all $k\in\{1,\dots,n\}$.
\end{lemma}
\begin{proof}
	From the definition of $e_k^{(n)}$ it follows by direct calculation that
	\begin{align*}
		\big\|e_k^{(n)}\big\|_\infty &< (2\pi n)^{-\f d2},\\
		\big\|\del_j e_k^{(n)}\big\|_\infty &< (2\pi)^{-\f d2}\f{n^{-\f d2+1}}{2}\left( \left((i_k)_j+\tfrac1n\right)^2 - (i_k)_j^2 \right)
	\end{align*}
	from which the assertion follows. Note that the bound in the second equation can be made independent of $k$, because $i_k\in L_n\subset B_n(0)$ for all $k$.
\end{proof}
The information accessible to the algorithm will be the set 
\begin{equation}\label{eq:def_Lambda_3}
\begin{split}
	\Lambda_3 &:= \left\{ \rho_x \,|\,x\in\R^d\right\}\cup \left\{e_k^{(n)}(i)\,\Big|\,i\in l^{-1}\mathbb Z^d,\, l\in\N,\;k\in\{1,\dots,n\},\;n\in\N\right\}\\
	&\qquad\qquad\qquad\qquad\qquad \cup \bigg\{\tfrac{n\delta_{mk}}{3}\sum_{j=1}^d\left( \left((i)_j+\tfrac1n\right)^3 - (i)_j^3 \right)\,\Big|\, i\in L_n,\; n\in\N  \bigg\}
\end{split}
\end{equation}
where $\rho_x(V)=V(x)$ are the evaluation functionals and $e_k^{(n)}(i)$ denote constant functions that map $V$ to the number $e_k^{(n)}(i)$. The meaning of the constants $\tfrac{n\delta_{mk}}{3}\sum_{j=1}^d\big(\big((i)_j+\tfrac1n\big)^3 - (i)_j^3 \big)$ will become clear later on. 

Together, $\Om_3$ and $\Lambda_3$ define a computational problem $(\Om_3,\Lambda_3,\sigma(\cdot))$. The main result of this section is the following.
\begin{theorem}\label{th:Schroedinger}
	For $\Om_3$ and $\Lambda_3$ defined as above, we have $\SCI\big(\Om_3,\Lambda_3,\sigma(\cdot)\big)=1$.
\end{theorem}
The proof of Theorem \ref{th:Schroedinger} will be by reduction to Theorem \ref{th:perturbation}. In order to accomplish this, we need to be able to compute the matrix elements $\left\langle (-\Delta+V)e_i,e_j \right\rangle$ \emph{by performing only a finite number of algebraic operations on a finite number of values of $V$}. This will be the main difficulty.

\subsection{Proof of Theorem \ref{th:Schroedinger}}

 We first note that by the unitarity of the Fourier transform, the $e_k^{(n)}$ still form an orthonormal basis of $\h_n$ and we still have $P_{\h_n}\to I$ strongly in $L^2(\R^d)$. The last claim follows immediately from the equality
\begin{align*}
	\left\|\sum_k\left\langle f, n^{\f d2}\chi_{i_k+[0,\f1n)^d}\right\rangle n^{\f d2}\chi_{i_k+[0,\f1n)^d} - f\right\|_{L^2(\R^d)} = \left\|\sum_k\left\langle \hat f, e_k\right\rangle e_k - \hat f\right\|_{L^2(\R^d)}.
\end{align*}
Next, we show that the spaces $\h_n$ defined in \eqref{eq:Hn_Def} are indeed a reasonable choice for the problem at hand. More precisely, we have
\begin{lemma}\label{lemma:H_n_is_core}
	The union $\bigcup_{n\in\N}\h_n$ is a core for $-\Delta$.
\end{lemma}
\begin{proof}
	By means of the Fourier transform the assertion is equivalent to the space $\bigcup_{n\in\N}\widehat\h_n$ being a core for the multiplication operator $u\mapsto |\xi|^2 u$ in $L^2(\R^d)$. To verify this, we have to show that for every $u\in\dom(|\xi^2|)$ there exists a sequence $u_n\in\h_n$ such that 
	\begin{enumi}
		\item $\|u_n-u\|_{L^2(\R^d)}\to 0,$
		\item $\left\||\xi|^2(u_n-u)\right\|_{L^2(\R^d)}\to 0$
	\end{enumi}
	Point (i) is easily shown by choosing
	\begin{align}\label{eq:u_n_def}
		u_n:=\sum_{i\in L_n}\left\langle u , n^{\f d2}\chi_{i+[0,\f1n)^d}\right\rangle n^{\f d2}\chi_{i+[0,\f1n)^d}.
	\end{align}
	Indeed, for smooth $u$ the $L^2$-convergence of $u_n$ to $u$ is standard, while the general case follows by a density argument. We omit the technical details. To show point (ii), let $R>0$ and decompose the norm in (ii) as
	\begin{align}\label{eq:B_R-decomposition}
		\left\||\xi|^2(u_n-u)\right\|_{L^2(\R^d)}^2 &= \int_{B_R} \left||\xi|^2(u_n-u)\right|^2\,d\xi
		+
		\int_{\R^d\setminus B_R} \left||\xi|^2(u_n-u)\right|^2\,d\xi,
	\end{align}
	where $B_R$ denotes the ball of radius $R$ centered at 0. We first estimate the second term on the right hand side. To this end, we let $u_n$ be defined by \eqref{eq:u_n_def} and employ the shorthand notation $\chi_i:=n^{\f d2}\chi_{i+[0,\f1n)^d}$. On the whole space we have
	\begin{align*}
		\left\| |\xi|^2u_n \right\|_{L^2(\R^d\setminus B_R)}^2 &= \left\| |\xi|^2\sum_{i\in L_n}\langle u,\chi_i\rangle\chi_i \right\|_{L^2(\R^d\setminus B_R)}^2\\
		&\leq \sum_{i\in L_n\setminus B_R} |\langle u,\chi_i\rangle|^2 \left\||\xi|^2\chi_i\right\|_{L^2(\R^d)}^2\\
		&\leq \sum_{i\in L_n\setminus B_R}\|u\|_{L^2(i+[0,\f1n)^d)}^2\left\|n^{\f d2}|\xi|^2\right\|_{L^2(i+[0,\f1n)^d)}^2,
	\end{align*}
	where we have used the fact that $\supp(\chi_i)\cap\supp(\chi_j)=\emptyset$ for $i\neq j$. The factor $\big\|n^{\f d2}|\xi|^2\big\|_{L^2(i+[0,\f1n)^d)}$ on the right hand side is clearly bounded by $\sup_{\xi\in i+[0,\f1n)^d}|\xi|^2$. Thus, if we define a function $F_n$ by
	\begin{align*}
		F_n(\xi):=\sum_{i\in\f1n\mathbb Z^d}\left(\sup_{\eta\in i+[0,\f1n)^d}|\eta|^2\right)\chi_{i},
	\end{align*} 
	then we will have (note that $F_n$ is constant on each of the cubes $i+[0,\f1n)^d$)
	\begin{align*}
		\left\| |\xi|^2u_n \right\|_{L^2(\R^d\setminus B_R)}^2
		&\leq \sum_{i\in L_n\setminus B_R}\|u\|_{L^2(i+[0,\f1n)^d)}^2 F_n(\xi)^2\\
		&= \sum_{i\in L_n\setminus B_R}\left\|F_n(\xi)u\right\|_{L^2(i+[0,\f1n)^d)}^2\\
		&\leq \left\|F_n(\xi)u\right\|_{L^2\big(\R^d\setminus B_{R-\f{\sqrt{d}}{n}}\big)}^2
	\end{align*}
	Next, we note that it is easy to see that there exist constants $a,b>0$ such that $F_n(\xi)\leq a|\xi|^2+b$ uniformly in $n$ (see Figure \ref{fig:xi^2}).
	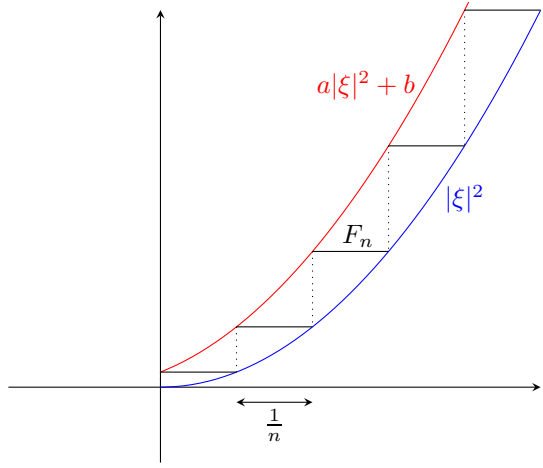
\begin{figure}
		\centering
		\begin{tikzpicture}[>=stealth, domain=0:5, smooth]
    \clip (-2,-1) rectangle (5.1,5.1);

	\draw[->] (-2,0) -- (5,0);
	\draw[->] (0,-1) -- (0,5);
	
	\draw[color=blue]   plot (\x,0.2*\x*\x);
	\draw[color=red]   plot (\x,0.2*\x*\x+0.4*\x+0.2);
	
	\foreach \x in {1,2,3,4,5}{
		\draw (\x,\x^2/5) -- (\x-1,\x^2/5);
		\draw[dotted] (\x-1,0.2*\x^2) -- (\x-1,0.2*\x^2-0.2*2*\x+0.2);
	}
	
	\draw (4,2.5) node{${\color{blue}|\xi|^2}$};
	\draw (2.7,4) node{${\color{red}a|\xi|^2+b}$};
	\draw (2.6,2) node{${\color{black}F_n}$};
	
	\draw[<->] (1,-0.2) -- (2,-0.2);
	\draw (1.5,-0.5) node{$\tfrac 1 n$};
\end{tikzpicture}
		\caption{Sketch of function $F_n$.}\label{fig:xi^2}
	\end{figure}
	Overall we conclude that
	\begin{align*}
		\left\| |\xi|^2u_n \right\|_{L^2(\R^d\setminus B_R)}^2
		&\leq \left\|F_n(\xi)u\right\|_{L^2\big(\R^d\setminus B_{R-\f{\sqrt{d}}{n}}\big)}^2\\
		&\leq \left\|(a|\xi|^2+b)u\right\|_{L^2(\R^d\setminus B_{R-1})}^2,
	\end{align*}
	where the last term on the right hand side is finite because by assumption $u\in\dom(|\xi|^2)$. In fact, from this last inequality we can see immediately that 
	\begin{align*}
		\left\| |\xi|^2u_n \right\|_{L^2(\R^d\setminus B_R)}^2\to 0
	\end{align*}
	as $R\to\infty$ uniformly in $n$. Estimating the second term on the right hand side of eq. \eqref{eq:B_R-decomposition} is now straightforward. We get
	\begin{align*}
		\int_{\R^d\setminus B_R} \left||\xi|^2(u_n-u)\right|^2\,d\xi &\leq 
		\left\| |\xi|^2u_n\right\|^2_{L^2(\R^d\setminus B_R)} + \left\| |\xi|^2u\right\|^2_{L^2(\R^d\setminus B_R)}\\
		&\leq \left\|(a|\xi|^2+b)u\right\|_{L^2(\R^d\setminus B_{R-1})}^2 + \left\| |\xi|^2u\right\|^2_{L^2(\R^d\setminus B_R)}.
	\end{align*}
	Now let $\eps>0$ and choose $R$ so large that $\big\|(a|\xi|^2+b)u\big\|_{L^2(\R^d\setminus B_{R-1})}^2 + \big\| |\xi|^2u\big\|^2_{L^2(\R^d\setminus B_R)}<\eps$. From eq. \eqref{eq:B_R-decomposition} we then see that
	\begin{align*}
		\limsup_{n\to\infty}\left\||\xi|^2(u_n-u)\right\|_{L^2(\R^d)}^2 &\leq \limsup_{n\to\infty}\int_{B_R} \left||\xi|^2(u_n-u)\right|^2\,d\xi+\eps\\
		&\leq \limsup_{n\to\infty}R^2\int_{B_R} \left|u_n-u\right|^2\,d\xi+\eps\\
		&= \eps,
	\end{align*}
	because $u_n\to u$ in $L^2(\R^d)$. Since $\eps$ was arbitrary, it follows that 
	\begin{align*}
		\limsup_{n\to\infty}\left\||\xi|^2(u_n-u)\right\|_{L^2(\R^d)}^2=0,
	\end{align*}
	which concludes the proof.
\end{proof}

Our strategy for proving Theorem \ref{th:Schroedinger} is as follows. By the assumptions on $V$ stated in the definition of $\Om_3$ and Lemma \ref{lemma:H_n_is_core} we know that we have $\Om_3\subset\Om_2$, if we choose $\h=L^2(\R^d)$ and $\h_n$ as in \eqref{eq:Hn_Def}. Hence, we already know from Theorem \ref{th:perturbation} that $\Gamma^{(2)}_n(H)\to\sigma(H)$ for all $H\in\Om_3$. However, $\Gamma^{(2)}_n$ uses the matrix elements $\big\langle H e_k^{(n)}, e_j^{(n)}\big\rangle$, which we are not allowed to access in Theorem \ref{th:Schroedinger}. Therefore, we will define a new algorithm $\Gamma^{(3)}_n$ which only accesses the information provided in $\Lambda_3$ and which satisfies $\Gamma^{(3)}_n(H)\approx\Gamma^{(2)}_n(H)$ for $H\in\Om_3$ in an appropriate sense.

\paragraph{The algorithm.}
As described above, we need to approximate the matrix elements $\langle-\Delta e_k^{(n)},e_m^{(n)}\rangle$ and $\langle Ve_k^{(n)},e_m^{(n)}\rangle$ using only a finite amount of information provided in the set $\Lambda_3$. We start with the Laplacian, which is the simpler case. Indeed, we have
\begin{align*}
	\left\langle-\Delta e_k^{(n)},e_m^{(n)}\right\rangle &= \left\langle |\xi|^2n^{\f d2}\chi_{i_k+[0,\f1n)^d} \,,\, n^{\f d2}\chi_{i_m+[0,\f1n)^d}\right\rangle\\
	&= n^d\delta_{mk}\int_{i_k+[0,\f1n)^d} |\xi|^2\,d\xi\\
	&= \f{n\delta_{mk}}{3}\sum_{j=1}^d\left( \left((i_k)_j+\f1n\right)^3 - (i_k)_j^3 \right).
\end{align*}
Note that these are precisely the terms in the third factor in eq. \eqref{eq:def_Lambda_3}. 

Next, we will compute the matrix elements $\langle Ve_k^{(n)},e_m^{(n)}\rangle$. Since any algorithm can only use finitely many values of $V$, we will have to perform an approximation procedure. To this end, let $l\in\N$ and define a lattice $P_l\subset \R^d$ by
\begin{align*}
	P_l:= \f1l\mathbb Z^d\cap Q_l(0),
\end{align*}
where $Q_l$ denotes the cube of edge length $l$ centered at 0. Next, let 
\begin{align*}
	V_l(x) := \sum_{i\in P_l} V(i)\chi_{i+[0,\f1l)^d}.
\end{align*}
\begin{lemma}\label{lemma:V-Vk}
	For any function $f\in C^1(\R^d)$ one has 
	\begin{align*}
		\left\|f-\sum_{i\in P_l} f(i)\chi_{i+[0,\f1l)^d}\right\|_{\infty}\leq \f{\|\nabla f\|_\infty}{l},
	\end{align*}
%	If $-\Delta+V\in\Om_3$, then
%	\begin{align*}
%		\|V-V_l\|_{\infty}\leq \f Ml,
%	\end{align*}
%	with $M$ as in eq. \eqref{eq:def_Omega3}.
\end{lemma}
\begin{proof}
	This follows immediately from the identity
	\begin{align*}
		f(x)-f(i) = \int_{[i,x]}\nabla f(t)\cdot dt,
	\end{align*}
	where $i\in P_l$ and $[i,x]$ denotes a line segment connecting $i$ to $x\in i+[0,\f1l)^d$.
\end{proof}
%By slight abuse of notation, we use the same letters $V,V_l$ to denote the corresponding multiplication operators on $L^2(\R^d)$. We can conclude immediately from Lemma \ref{lemma:V-Vk} that the corresponding operators are close in norm.
%\begin{corollary}
%	If $-\Delta+V\in\Om_3$, then
%	\begin{align*}
%		\|V-V_l\|_{L(L^2(\R^d))}<\f Ml.
%	\end{align*}
%\end{corollary}
In order to define our approximation of $\langle Ve_k^{(n)},e_m^{(n)}\rangle$, we additionally introduce the step function approximation
\begin{align*}
	E_{k,l}(x) := \sum_{i\in P_l} e_k^{(n)}(l)\chi_{i+[0,\f1l)^d}.
\end{align*}
%As in Lemma \ref{lemma:V-Vk} we have
%\begin{align}
%	\left\|E_{k,l}-e_k^{(n)}\right\|_\infty \leq \f{1}{l}(2\pi)^{-\f d2}dn^{3-\f{d}{2}},
%\end{align}
%(cf. Lemma \ref{lemma:e_k-Bound}) for all $k$. 
\begin{lemma}\label{lemma:V_ij-Vl_ij}
	For $-\Delta+V\in\Om_3$ one has 
	\begin{align*}
		\left|\left\langle Ve_k^{(n)},e_m^{(n)}\right\rangle - \left\langle V_l E_{k,l},E_{m,l}\right\rangle\right|\leq  \f{3|C|}{l}M(2\pi)^{-\f d2}n^{3-\f{d}{2}}d,
	\end{align*}
	for all $l>2\diam(C)$, where $M$ and $C$ are as in eq. \eqref{eq:def_Omega3}.
\end{lemma}
\begin{proof}
	By assumption $l$ is large enough such that $C\subset Q_l(0)$ (with $C$ from eq. \eqref{eq:def_Omega3}). We calculate the error
\begin{align*}
	\left|\left\langle Ve_k^{(n)},e_m^{(n)}\right\rangle - \left\langle V_l E_{k,l},E_{m,l}\right\rangle\right| &= \left|\sum_{i\in P_l}\int_{i+[0,\f1l)}Ve_k^{(n)}e_m^{(n)}\,dx - \sum_{i\in P_l}\int_{i+[0,\f1l)}V_lE_{k,l}E_{m,l}\,dx\right|\\
	&\leq \sum_{i\in P_l}\int_{i+[0,\f1l)}\left|Ve_k^{(n)}e_m^{(n)} - V_lE_{k,l}E_{m,l}\right|dx\\
	&\leq \sum_{i\in P_l}\int_{i+[0,\f1l)} l^{-1}\left\|\nabla\big(Ve_k^{(n)}e_m^{(n)}\big)\right\|_\infty\chi_C\, dx\\
	&= l^{-1} |C| \left\|\nabla\big(Ve_k^{(n)}e_m^{(n)}\big)\right\|_\infty\\
	&\leq  \f{|C|}{l} \left(\left\|\nabla Ve_k^{(n)}e_m^{(n)}\right\|_\infty\! +\left\|V\nabla\big(e_k^{(n)}\big)e_m^{(n)}\right\|_\infty\! + \left\|Ve_k^{(n)}\nabla\big(e_m^{(n)}\big)\right\|_\infty\right)\\
	&\leq  \f{3|C|}{l}M(2\pi)^{-\f d2}n^{3-\f{d}{2}}d,
\end{align*}
where we have used Lemma \ref{lemma:V-Vk} in the third line and Lemma \ref{lemma:e_k-Bound} and the fact that $\|V\|_{\mathcal C^1}\leq M$ in the last line.
\end{proof}
\begin{corollary}
	If we denote $H_n:=P_n(-\Delta+V)|_{\h_n}$ and $H_n^l:= P_n(-\Delta+V_l)|_{\h_n}$, then 
	\begin{align*}
		\left\| H_n - H_n^l \right\|_{L(\h_n)} \leq \f{3|C|}{l}M(2\pi)^{-\f d2}n^{4-\f{d}{2}}d.
	\end{align*}
\end{corollary}
\begin{proof}
	By Lemma \ref{lemma:V_ij-Vl_ij} the matrix elements of $H_n$ and $H_n^l$ satisfy $\big|(H_n)_{km} - (H_n^l)_{km}\big|\leq \f{3|C|}{l}M(2\pi)^{-\f d2}n^{3-\f{d}{2}}d$. Now note that for any two matrices $A=(A_{km})$ and $B=(B_{km})$ one has
	\begin{align*}
		\left\|(A-B)x\right\|^2_{\h_n} &= \sum_{k=1}^n\left|\sum_{m=1}^n (A_{km}-B_{km})x_m\right|^2\\
		&\leq \left(\sup_{k,m}|A_{km}-B_{km}|\right)\sum_{k,m=1}^n|x_m|^2\\
		&= n\left(\sup_{k,m}|A_{km}-B_{km}|\right)\|x\|_{\h_n}^2.
	\end{align*}
	This immediately implies the assertion.
\end{proof}
We are finally ready to define our algorithm. Let $n\in\N$ and choose $l(n)\in\N$ large enough such that $\f{3|C|}{l(n)}M(2\pi)^{-\f d2}n^{4-\f{d}{2}}d<\f{1}{2n}$. Define for $H=-\Delta+V\in\Om_3$
\begin{align*}
	\Lambda_{\Gamma_n^{(3)}}(H) &:= \left\{ \rho_i \,|\,i\in P_{l(n)}\right\}\cup \left\{e_k^{(n)}(i)\,\Big|\,i\in P_{l(n)},\;k\in\{1,\dots,n\}\right\}\\
	&\qquad\qquad\qquad\qquad\qquad \cup \bigg\{\tfrac{n\delta_{mk}}{3}\sum_{j=1}^d\left( \left((i)_j+\tfrac1n\right)^3 - (i)_j^3 \right)\,\Big|\, i\in L_n  \bigg\}\\
	\Gamma_n^{(3)}(H) &:= \left\{ \lambda\in G_n\,\bigg|\, \left\|(H_n^{l(n)}-\lambda)^{-1}\right\|^{-1} \leq \f2n \right\}\cup\Gamma_n^{(1)}(-\Delta)
\end{align*}
with the convention that $\|(H_n^{l(n)}-\lambda)^{-1}\|^{-1}=0$ when $\lambda\in\sigma(H_n^{l(n)})$. Note that $\Lambda_{\Gamma_n^{(3)}}(H)$ is a finite set for each $H\in\Om_3$ and by Lemma \ref{lemma:finite_testing} determining whether $\left\|(H_n^{l(n)}-\lambda)^{-1}\right\|^{-1} \leq \f2n$ requires only finitely many algebraic operations. Moreover, since $\Lambda_3$ contains all matrix elements of the Laplacian, we conclude that computing $\Gamma_n^{(1)}(-\Delta)$ can also be done by performing a finite amount of algebraic operations on the information provided. Overall, we conclude that each $\Gamma_n^{(3)}$ is an arithmetic algorithm in the sense of Definition \ref{def:Algorithm}. 

\paragraph{Convergence.}
It remains to prove that $\Gamma_n^{(3)}(H)\to\sigma(H)$ in the Attouch-Wets metric. To this end, let $\lambda\in G_n$ and note that by the second resolvent identity we have
\begin{align*}
	(H_n^{l(n)}-\lambda)^{-1} - (H_n-\lambda)^{-1} = (H_n^{l(n)}-\lambda)^{-1}(H_n-H_n^{l(n)})(H_n-\lambda)^{-1}
\end{align*}
Taking norms on both sides and using the reverse triangle inequality we obtain\footnote{
	For $\lambda\in\rho(H_n)\cap\rho\big(H_n^{l(n)}\big)$ this inequality follows directly from the second resolvent identity, while for $\lambda\in\sigma(H_n)\cup\sigma\big(H_n^{l(n)}\big)$ it is shown by a Neumann series argument.
}
\begin{align*}
	\left|\|(H_n^{l(n)}-\lambda)^{-1}\|^{-1} - \|(H_n-\lambda)^{-1}\|^{-1}\right| &\leq \|H_n-H_n^{l(n)}\|\\
	&\leq \f{3|C|}{l(n)}M(2\pi)^{-\f d2}n^{4-\f{d}{2}}d\\
	&\leq \f{1}{2n},
\end{align*}
where the last line follows from our choice of $l(n)$. Now, if $\lambda\in\Gamma_n^{(3)}(H)$ the above inequality implies that 
\begin{align*}
	\|(H_n-\lambda)^{-1}\|^{-1} &\leq \|(H_n^{l(n)}-\lambda)^{-1}\|^{-1}+\f{1}{2n}\\
	&\leq \f 2n +\f{1}{2n}\\
	&\leq \f{3}{n}
\end{align*}
and hence $\lambda\in\Xi_n(H)$ (cf. Remark \ref{Remark}). Similarly, if $\lambda\in\Gamma^{(2)}_n(H)$ then
\begin{align*}
	\|(H_n^{l(n)}-\lambda)^{-1}\|^{-1}&\leq \|(H_n-\lambda)^{-1}\|^{-1} + \f{1}{n}\\
	&\leq \f {1}{n} +\f{1}{2n}\\
	&\leq \f{2}{n}
\end{align*}
and hence $\lambda\in\Gamma_n^{(3)}(H)$. Thus, we have the inclusions
\begin{align*}
	\Gamma^{(2)}_n(H) \subset \Gamma_n^{(3)}(H) \subset \Xi_n(H).
\end{align*}
Since $\Gamma^{(2)}_n(H)\to\sigma(H)$ and $\Xi_n(H)\to\sigma(H)$ by Theorem \ref{th:perturbation} and Remark \ref{Remark}, we conclude that $ \Gamma_n^{(3)}(H)\to\sigma(H)$ as well. This concludes the proof of Theorem \ref{th:Schroedinger}.

%\begin{thebibliography}{20}
%	\bibitem[BMT19]{BMC}
%	    S. B\"ogli, M. Marletta, C. Tretter.
%	    \newblock The Essential Numerical Range for Unbounded Operators.
%		\newblock {\em In preparation.}
%	\bibitem[B18]{B18}
%	    S. B\"ogli.
%	    \newblock Local Convergence of Spectra and Pseudospectra.
%		\newblock {\em J. Spectr. Theory} 8 (2018), 1051-1098.
%	\bibitem[AHS15]{AHS}
%		J. Ben-Artzi, A.C. Hansen, O. Nevanlinna, M. Seidel.
%		\newblock Can everything be computed? - On the Solvability Complexity Index and Towers of Algorithms.
%		\newblock \emph{Preprint,} arXiv:1508.03280.
%	\bibitem[RS80]{RS}
%		M.~{Reed} \& B.~{Simon}.
%		\newblock {\em {M}ethods of {M}odern {M}athematical {P}hysics {I}: {F}unctional
%  {A}nalysis}.
%  \bibitem[EE87]{EE}
%		D.~E. Edmunds and W.~D. Evans.
%		\newblock {\em {S}pectral {T}heory and {D}ifferential {O}perators ({O}xford {M}athematical {M}onographs)}.
%		\newblock {O}xford {U}niversity {P}ress, 1987.
%\end{thebibliography}

\nocite{*}
\bibliography{mybib.bib}
\bibliographystyle{alphaabbr}

\end{document}